\documentclass[preprint,1p]{elsarticle}

\makeatletter
 \def\ps@pprintTitle{%
 	\let\@oddhead\@empty
 	\let\@evenhead\@empty
 	\def\@oddfoot{\footnotesize\itshape
 		{} \hfill\today}%
 	\let\@evenfoot\@oddfoot
 }
\makeatother
\usepackage[unicode]{hyperref}

\usepackage{latexsym}
\usepackage{indentfirst}
\usepackage{amsxtra}
\usepackage{amssymb}
\usepackage{amsthm}
\usepackage{amsmath}
\usepackage{mathrsfs} 

\usepackage{xcolor}
\usepackage{color}

\usepackage{amsfonts}

\usepackage[capitalise]{cleveref}

\newtheorem{theor}{Theorem}
\newtheorem*{theor*}{Theorem}
\newtheorem{prop}[theor]{Proposition}
\newtheorem{lemma}[theor]{Lemma}
\newtheorem{cor}[theor]{Corollary}
\newtheorem*{cor*}{Corollary}
\theoremstyle{definition}               
\newtheorem{defin}[theor]{Definition}
\newtheorem{ex}{Example}
\newtheorem{exs}[ex]{Examples}
\newtheorem{rem}{Remark}
\newtheorem{rems}[rem]{Remarks}

\newcommand{\seq}[1]{\left<#1\right>}

\DeclareMathOperator{\Sym}{Sym}
\DeclareMathOperator{\Aut}{Aut}
\DeclareMathOperator{\End}{End}
\DeclareMathOperator{\id}{id}

\DeclareMathOperator{\E}{E}

\newcommand{\lambdaa}[2]{\lambda_{#1}{#2}}

\newcommand{\rhoo}[2]{\rho_{#1}{#2}}
\newcommand{\alphaa}[3]{\alpha^{#1}_{#2}{#3}}

\usepackage[makeroom]{cancel}
 
\usepackage{soul}
\usepackage{tikz}

\begin{document}

\begin{frontmatter}
	\title{Set-theoretic solutions of the Yang-Baxter equation associated to weak braces\tnoteref{mytitlenote}}
	\tnotetext[mytitlenote]{This work was partially supported by the Dipartimento di Matematica e Fisica ``Ennio De Giorgi'' - Università del Salento. The second and the fourth authors were partially supported by the ACROSS project ARS01\_00702. The authors are members of GNSAGA (INdAM).}
	\author[unile]{F.~CATINO}
	\ead{francesco.catino@unisalento.it}
		\author[unile]{M.~MAZZOTTA}
	\ead{marzia.mazzotta@unisalento.it}
	\author[unile]{M.M.~MICCOLI}
	\ead{maddalena.miccoli@unisalento.it}
	\author[unile]{P.~STEFANELLI}
	\ead{paola.stefanelli@unisalento.it}
	\address[unile]{Dipartimento di Matematica e Fisica ``Ennio De Giorgi''
		\\
		Universit\`{a} del Salento\\
		Via Provinciale Lecce-Arnesano \\
		73100 Lecce (Italy)\\}

\begin{abstract} 
    We investigate a new algebraic structure which always gives rise to a set-theoretic solution of the Yang-Baxter equation. Specifically, a \emph{weak (left) brace} is a non-empty set $S$ endowed with two binary operations $+$ and $\circ$ such that both $(S,+)$ and $(S, \circ)$ are inverse semigroups and 
    \begin{align*}
            a \circ \left(b+c\right) = \left(a\circ b\right) - a + \left(a\circ c\right)\qquad \text{and} \qquad 
        a\circ a^- = - a + a
    \end{align*}
    hold, for all $a,b,c \in S$, where $-a$ and $a^-$ are the inverses of $a$ with respect to $+$ and $\circ$, respectively. In particular, such structures include that of skew braces and form a subclass of inverse semi-braces.
    Any solution $r$ associated to an arbitrary weak brace $S$ has a behavior close to bijectivity, namely $r$ is a completely regular element in the full transformation semigroup on $S\times S$. In addition, we provide some methods to construct weak braces.
\end{abstract}

\end{frontmatter}

\section*{Introduction}
The quantum Yang-Baxter equation is a fundamental topic of theoretical physics that appeared at first in \cite{Ya67} and, independently, in \cite{Ba72}. Drinfel'd \cite{Dr92} posed the question of finding all the set-theoretical solutions of this equation. Determining all the solutions, up to equivalence, is  very challenging and a large number of works related to this matter of study has been produced  over the years. 
Among the seminal papers, we mention those of Gateva-Ivanova and Van den Bergh \cite{GaVa98}, Gateva-Ivanova and Majid \cite{GaMa08}, Etingof, Schedler, and Soloviev \cite{ESS99}, Lu, Yan, and Zhu
 \cite{LuYZ00}.
For more details on the development of the studies to date, we refer the reader to the introduction of \cite{CaMaSt21} and references therein. If $S$ is a set, a map  $r:S\times S \longrightarrow S\times S$ satisfying the relation
\begin{align*}
    \left(r \times \id_S\right) 
    \left(\id_S \times r\right)
    \left(r \times \id_S\right) 
    = \left(\id_S \times r\right)
    \left(r \times \id_S\right)
    \left(\id_S \times r\right)
\end{align*}
is said to be a \emph{set-theoretic solution of the Yang-Baxter equation}, or shortly a \emph{solution}, on $S$.  
If $S$ and $T$ are sets, two solutions $r$ and $s$ on $S$ and $T$, respectively, are called \emph{equivalent} if there exists a bijective map $f:S \to T$ such that $(f \times f)r = s(f \times f)$, see \cite{ESS99}.
Given a solution $r$, it is usual to write  $r(a,b) = (\lambda_a(b),\rho_b(a))$, with $\lambda_a$ and $\rho_b$ maps from $S$ into itself, for all $a,b \in S$. Moreover, $r$ is said to be \emph{left non-degenerate} if $\lambda_a$ is bijective, for every $a\in S$, \emph{right non-degenerate} if $\rho_b$ is bijective, for every $b\in S$, and \emph{non-degenerate} if $r$ is both left and right non-degenerate.
If $r$ is neither left nor right non-degenerate, then it is called \emph{degenerate}. Furthermore, $r$ is called \emph{involutive} if $r^2 = \id_{S\times S}$, \emph{idempotent} if $r^2 = r$, and \emph{cubic} if $r^3 = r$.

A productive research line for studying solutions drawn by Rump \cite{Ru07} is based on the study of the theory of \emph{(left) braces}, algebraic structures which include the Jacobson radical rings. Guarnieri and Vendramin \cite{GuVe17} introduced a generalization of braces, namely, \emph{skew (left) braces}, that are triples $(S,+,\circ)$, where $(S,+)$ and $(S,\circ)$ are groups and the following condition holds
\begin{equation*}
	a\circ (b+c) = \left(a\circ b\right) - a + \left(a\circ c\right),
\end{equation*}
for all $a,b,c\in S$.
In particular, if the group
$\left(S,+\right)$ is abelian then $S$ is a brace.\\
\textbf{Convention.} To avoid parentheses, from now on, we will assume that the multiplication has higher precedence than the addition.\\
As proved in \cite[Theorem 3.1]{GuVe17} every skew brace gives rise to a bijective non-degenerate solution. As shown in \cite[Theorem 9]{CaCoSt17}, such a solution $r:S\times S\to S\times S$ can be written as
\begin{align}\label{map}
    r\left(a,b\right)
    = \left(a\circ \left(a^-+b\right),\ \left(a^-+b\right)^-\circ  b\right), 
\end{align}
for all $a,b \in S$, where $a^-$ is the inverse of $a$ in the group $\left(S, \circ\right)$. Moreover, $r$ is involutive if and only if $S$ is a brace. In the last years, several results about skew braces have been provided in \cite{Ch18,CCoSt19,Ze19,JeKuVaVe19,CeSmVe19,KoTr20,GoNa21}, just to recall a few.

To determine new solutions, in \cite{CaMaSt21} there were introduced inverse semi-braces, structures including skew braces. A\emph{(left) inverse semi-brace} is a triple $(S,+, \circ)$  such that $(S, +)$ is an arbitrary semigroup, $(S, \circ)$ is an inverse semigroup, and
\begin{align*}
  a\circ\left(b + c\right)
  = a\circ b + a\circ\left(a^- + c\right)
\end{align*}
holds, for all $a,b,c\in S$, where $a^-$ is the inverse of $a$ with respect to the multiplication. We recall that a semigroup $(S, \circ)$ is called \emph{inverse} if, for each $a \in S$, there exists a unique $a^- \in S$ satisfying $a\circ a^-\circ a=a$ and $a^-\circ a\circ a^-=a^-$. 
The books of Clifford and Preston \cite{ClPr61}, Petrich \cite{Pe84}, Howie \cite{Ho95}, and Lawson \cite{La98} contain the most important known results on the wide theory of inverse semigroups.\\
Inverse semi-braces $S$ with $(S,\circ)$ group are the semi-braces, structures studied in \cite{JeAr19}. They were initially introduced in  \cite{CaCoSt17}  in the cancellative case, namely when $(S,+)$ is a cancellative semigroup, and
recently deepened in \cite{CaCeSt20x}. Moreover, inverse semi-braces with $(S,\circ)$ a Clifford semigroup, i.e., an inverse semigroup having central idempotents, are the generalized semi-braces \cite{CCoSt20x-2}.\\ 
Note that if $S$ is an inverse semi-brace, the map $r$ given in \eqref{map}
is not necessarily a solution. Sufficient conditions so that it is can be found in \cite[Theorem 7]{CaMaSt21}.
Furthermore, in the context of semi-braces, a characterization has been given in \cite[Theorem 3]{CaCoSt19}. In particular, if $(S,+)$ is a left cancellative semigroup, the map $r$ is a left non-degenerate solution  \cite[Theorem 9]{CaCoSt17}.  

\smallskip

This paper aims to investigate a subclass of inverse semi-braces $S$ in which the map $r$ always is a solution. Specifically, a non-empty set $S$ endowed with two operations $+$ and $\circ$ is said to be a \emph{weak (left) semi-brace} if $(S,+)$ and $(S, \circ)$  both are inverse semigroups satisfying
\begin{align}\label{orig}
a\circ\left(b + c\right)
  = a\circ b + a\circ\left(a^- + c\right)\qquad \text{and}\qquad  a \circ \left(a^-+b\right) = -a + a\circ b,
    \end{align}
for all $a,b,c \in S$, where $-a$ denotes the inverse of $a$ with respect to the sum. We show that the conditions in \eqref{orig} are equivalent to
\begin{align}
a \circ \left(b+c\right)=a\circ b-a+a\circ c\qquad \text{and}\qquad  a \circ a^- = -a + a,
    \end{align}
for all $a,b,c \in S$, which make these structures very close to skew braces.\\
We highlight that the additive semigroup $(S,+)$ and the multiplicative semigroup $(S,\circ )$ have the same set of idempotents and besides $(S,+)$ is a Clifford semigroup. Fully exploiting the inverse semigroup on the additive structure, we recover most of the properties already known for skew braces.
Among these, the map $\lambda:S \to \End(S,+), \, a \mapsto \lambda_a$ is a homomorphism from $(S, \circ)$ into $\End(S,+)$ and the map $\rho:S \to \mathcal{T}_{S}, \ b \to \rho_b$ is an anti-homomorphism from $(S, \circ)$ into the  full transformation semigroup $\mathcal{T}_{S}$. In addition, the following relation holds
\begin{align*}
    a \circ b= a \circ \left(a^-+b \right)\circ \left(a^-+b \right)^-\circ b,
\end{align*}
for all $a,b \in S$, trivially satisfied in the case of $(S, \circ)$ a group, and which is crucial to prove our main result.\\
As mentioned before, the solution $r$ associated to a skew brace $S$ is bijective. In particular, $r^{-1}$ is the solution associated to the opposite skew brace $S^{op}$ defined by considering the opposite sum of $S$, see \cite[Proposition 3.1]{KoTr20}. Moreover, also maps $\lambda_a$ and $\rho_b$ are bijective. Any weak brace $S$ has an analogous behavior. Indeed, the maps $\lambda_a$ and $\rho_b$ admit an inverse in $\mathcal{T}_{S}$ given by $\lambda_{a^-}$ and $\rho_{b^-}$, respectively. Moreover, $r$ has a behavior near to bijectivity, since it is a completely  regular element in $\mathcal{T}_{S\times S}$, namely, it admits the map $r^{op}$ as an inverse, that is the solution associated to the opposite weak brace of $S$, and it holds $rr^{op} = r^{op}r$. 
\smallskip

Finally, we review some of the constructions of inverse semi-braces provided in \cite{CaMaSt21} to obtain several instances of weak braces. Moreover, classes of examples can be obtained starting from exactly factorizable Clifford monoids. For a fuller treatment we refer the reader to the survey \cite{Fi10} that contains methods for constructing factorizable inverse monoids. 
Specifically, our examples include those of skew braces provided by Guarnieri and Vendramin in \cite[Example 1.6]{GuVe17} and also by Smoktunowicz and Vendramin in \cite[Theorem 3.3]{SmVe18}, which were mainly motivated by the result of Weinstein and Xu in \cite[Theorem 9.2]{WeXu92}. In addition, we review some of the constructions of inverse semi-braces provided in \cite{CaMaSt21}, to produce other new examples.

\bigskip

\section{Basics on inverse semi-braces}
In this section, we recall some basics on  inverse semi-braces and we give some new properties related to them. 
\medskip

For the ease of the reader, we initially recall essential notions for
our treatment which one can find, for instance, in \cite{Ho95} and \cite{Pe84}. Given a semigroup $(S, \circ)$ and $a \in S$, we say that an element $x$ of $S$ is an \emph{inverse} of $a$ if $a\circ x\circ a=a$ and $x\circ a\circ x=x$ hold. The semigroup $(S,\circ)$ is called \emph{inverse} if, for each $a\in S$, there exists a unique inverse of $a$, which we denote by $a^-$. Clearly, every group is an inverse semigroup.  A fundamental instance of inverse semigroup is the set $\mathcal{I}_X$ consisting of all the partial one-to-one maps of a non-empty set $X$ under the standard operation $\circ$ of composition of relations
(see \cite[Theorem  5.1.5]{Ho95}).\\
The behavior  of inverse elements in an inverse semigroup $(S,\circ)$ is similar to that in a group, since $(a\circ b)^{-}=b^{-}\circ a^{-}$ and $(a^{-})^{-}=a$, for all $a,b \in S$. If $a \in S$, then $a\circ a^{-}$ and $a^{-}\circ a$ are idempotents of $S$. Moreover, the set $\E(S,\circ)$ of the idempotents  is a commutative subsemigroup of $S$ and $e=e^-$, for every $e \in \E(S,\circ)$.\\ 
In addition, inverse semigroups in which its idempotents  are central is called a \emph{Clifford semigroups}. Such semigroups belong to the class of \emph{completely regular semigroups}, namely semigroups $(S,\circ)$ for which, for any $a\in S$, there exists a unique inverse $a^{-1}$ of $a$ such that $a\circ a^{-1}= a^{-1}\circ a$. Note that in a Cifford semigroup $a^{-}=a^{-1}$, for every $a\in S$.

\medskip
\begin{defin}[Definition 3, \cite{CaMaSt21}]\label{def_inverse}
    Let $S$ be a non-empty set endowed with two operations $+$ and $\circ$ such that $\left(S,+\right)$ is a semigroup (not necessarily commutative) and $\left(S,\circ\right)$ is an inverse semigroup. Then, we say that $\left(S, + , \circ \right)$ is \emph{a (left) inverse semi-brace} if the following holds
    \begin{align*}
        a\circ \left(b+c\right) = a\circ b + a\circ \left(a^{-} + c\right),
    \end{align*}
    for all $a, b, c \in S$, where $a^-$ is the inverse of $a$ with respect to $\circ$. 
    We call $(S,+)$ and $(S,\circ )$ the \emph{additive semigroup} and the \emph{multiplicative semigroup} of $S$, respectively.
\end{defin}
  
\noindent Any semi-brace \cite{CaCoSt17,JeAr19} is an inverse semi-brace, since in this case $\left(S,\circ\right)$ is a group.
Other examples of inverse semi-braces are generalized semi-braces \cite{CCoSt20x-2},  with $\left(S,\circ\right)$ a Clifford semigroup. By the way, we recall that a \emph{generalized (left) semi-brace} is a structure $(S,+,\circ)$ 
such that $(S,+)$ is a semigroup, $(S,\circ)$ is a completely regular semigroup, and it holds $a\circ \left(b+c\right) = a\circ b + a\circ \left(a^{-1} + c\right)$, for all $a,b,c\in S$.\\
Any skew brace \cite{GuVe17} and, in particular any brace \cite{Ru07}, is an inverse semi-brace, too. Indeed, $- a + a\circ b = a\circ a^- - a + a\circ b = a\circ \left(a^- + b\right)$, for all $a, b\in S$. 

\medskip

Given an inverse semi-brace $S$, the map $r:S\times S\to S\times S$ given by
\begin{align*}
    r(a,b) = \left(a\circ \left(a^{-} + b\right), \; \left(a^{-} + b\right)^-\circ b\right), 
\end{align*}
for all $a,b \in S$, is called the \emph{map associated} to $S$. Sufficient conditions so that the map $r$ is a solution have been provided in \cite[Theorem 7]{CaMaSt21}. For the class of semi-braces, one can find a characterization in \cite[Theorem 3]{CaCoSt19}. In particular, if $(S,+)$ is a left cancellative semigroup, the map $r$ is a left non-degenerate solution, see \cite[Theorem 9]{CaCoSt17}. In the more specific case of $S$ a skew brace, $r$ is a solution which is bijective and non-degenerate, as proved in \cite[Theorem 3.1]{GuVe17}.

\medskip

We point out that, every non-commutative Clifford semigroup gives rise to two inverse semi-braces that produce two not equivalent solutions.
\begin{ex}\cite{CaMaSt21}\label{exs-1}
Let $(S, \circ)$ be an arbitrary Clifford semigroup and let us consider the \emph{trivial inverse semi-brace} $(S,+,\circ)$, with $a+b = a \circ b$, and the \emph{almost trivial inverse semi-brace} $\left(S, \tilde{+},\circ\right)$, with $a \ \tilde{+}\ b = b \circ a$, for all $a,b \in S$.
Then, the solutions $r$ and $s$ associated to $(S,+,\circ)$ and $\left(S, \tilde{+},\circ\right)$ are
\begin{align*}
 r\left(a , b\right)=\left(a \circ a^-\circ b, \ b^-\circ a \circ b\right)   \qquad s\left(a , b\right)=\left(a \circ b\circ a^-, \ a \circ  b^-\circ b\right),
\end{align*}
for all $a, b \in S$, respectively. Moreover, if $(S, \circ)$ is commutative, the maps $r$ and $s$ coincide and $r$ is cubic.
\end{ex}
\noindent Note that the previous examples include \cite[Example 1.3]{GuVe17}, given in the context of skew braces. In addition, if the group $\left(S,\circ\right)$ is not abelian, $r=s^{-1}$, thus such solutions are not equivalent.

\medskip

As usual, if $S$ is an arbitrary inverse semi-brace, let $\lambda: S\to \End(S,+), \,a\mapsto\lambda_a$ and $\rho: S\to \mathcal{T}_S, \,b\mapsto\rho_b$ be the maps defined by 
\begin{align*}
   \lambda_a\left(b\right) = a\circ \left(a^{-} + b\right) \qquad \rho_b\left(a\right) = \left(a^{-} + b\right)^{-}\circ b, 
\end{align*}
for all $a,b\in S$, respectively.

\medskip

\noindent
The following proposition contains some properties which essentially involve the map $\lambda_a$ and the idempotents of the multiplicative semigroup of any inverse semi-brace. 
\begin{prop}\label{prop_semi_inversi}
  Let $S$ be an inverse semi-brace. Then, the following assertions hold:
  \begin{enumerate}    
    \item $\lambda_{a \circ a^-}=\lambda_{a}\lambda_{a^-}$,
    \item $\lambda_{a \circ b \circ b^-}=\lambda_{a \circ b}\lambda_{b^-}$,
    \item $a \circ \lambda_{a^-}(b)=a+\lambda_{a\circ a^-}(b)$,
          \item $a+\lambda_a(b)=a \circ \left(a^-\circ a+b\right),$
  \end{enumerate}
  for all $a, b \in S$.
  \begin{proof}
  $1.$\ Let $a,b\in S$. Then, we get
\begin{align*}
    \lambda_{a\circ a^-}\left(b\right)
    &= a\circ a^-\circ\left(a\circ a^- + b\right)
    = a\circ \left(a^-\circ a\circ a^- + \lambda_{a^-}\left(b\right)\right)\\
    &= a\circ \left(a^- + \lambda_{a^-}\left(b\right)\right)
    = \lambda_{a}\lambda_{a^-}\left(b\right).
\end{align*}
$2.$  \ If $a,b,c\in S$, we obtain
\begin{align*}
    \lambda_{a \circ b \circ b^-}\left(c\right)
    = a\circ b \circ b^-\circ \left(b\circ b^- \circ a^-+c\right)=a \circ b \circ \left(b^- \circ a^- +\lambda_{b^-}\left(c\right)\right)
    = \lambda_{a \circ b}\lambda_{b^-}\left(c\right).
\end{align*}
$3.$ \ If $a,b\in S$, we have that \begin{align*}
a\circ\lambda_{a^-}(b)
= a\circ a^-\circ\left(a + b\right)
= a + \lambda_{a\circ a^-}\left(b\right).
\end{align*}
$4.$ \ If $a,b \in S$, it follows that
\begin{align*}
    a + \lambda_a(b)
    = a\circ a^-\circ a + a\circ\left(a^- + b\right)
    = a\circ\left(a^-\circ a + b\right),
\end{align*}
which completes the claim.
  \end{proof}
\end{prop}

\medskip

\noindent As a direct consequence, we get the following result.
\begin{cor}\label{lambdaalambaa-}
Let $S$ be an inverse semi-brace. Then, the following equalities
\begin{align*}
    \lambda_{a}\lambda_{a^-}\lambda_{a} = \lambda_{a}
    \qquad \text{and} \qquad
    \lambda_{a^-}\lambda_{a}\lambda_{a^-} = \lambda_{a^-}
\end{align*}
hold, for any $a\in S$.
\end{cor}

\smallskip

\noindent In other words,  for every $a \in S$, the map $\lambda_a$ admits the map $\lambda_{a^{-}}$ as an inverse in $\mathcal{T}_S$. 

\medskip

Finally, we note that in the case of $(S, \circ)$ a group, it trivially holds the equality $a \circ b=\lambda_a(b) \circ \rho_b(a)$, for all $a,b \in S$. More generally, arbitrary inverse semi-braces satisfy the following weaker properties. 
\begin{prop}\label{prop_inversi}
  Let $S$ be an inverse semi-brace. Then, the following statements hold:
\begin{enumerate}
    \item $\lambda_a(b)^-\circ \lambda_a(b) \circ \rho_b(a)=\lambda_a(b)^- \circ a \circ b$,
          \item $ \lambda_a(b) \circ \rho_b(a)\circ \rho_b(a)^-= a \circ b\circ \rho_b(a)^-$,
\end{enumerate}
 for all $a, b \in S$.     
 \begin{proof}
We only prove $1.$ and, in the same way, one can check the second equality. If $a,b \in S$, then
\begin{align*}
    \lambda_a(b)^-\circ \lambda_a(b) \circ \rho_b(a)&=\left(a^-+b\right)^-\circ a^- \circ a \circ \left(a^-+b\right) \circ \left(a^-+b\right)^-\circ b\\     &=\left(a^-+b\right)^-\circ \left(a^-+b\right) \circ \left(a^-+b\right)^-\circ a^- \circ a \circ b\\
     &=\left(a^-+b\right)^-\circ a^- \circ a \circ b=\lambda_a(b)^- \circ a \circ b.
\end{align*}
Therefore, the assertion is proved.
 \end{proof}
\end{prop}

\bigskip

\section{Weak braces}

In this section, we introduce the algebraic structure of weak brace which generalizes that of skew brace and forms a subclass of inverse  semi-braces. Examples and several structural properties are given. 
\medskip

From now on, if $\left(S,+\right)$ and $\left(S,\circ\right)$ are inverse semigroups, for any $a \in S$, we denote by $a^-$ and $-a$ the inverses of $a$ with respect to $\circ$ and $+$, respectively. 
\begin{defin}\label{def:skew-inverse}
    Let $S$ be a non-empty set endowed with two operations $+$ and $\circ$ such that $\left(S,+\right)$ and $\left(S,\circ\right)$ both are inverse semigroups. Then, $(S, +, \circ)$ is a \emph{weak (left) brace} if the following relations hold
 \begin{align*}
    a\circ\left(b + c\right)
    = a\circ b - a + a\circ c\qquad \text{and}\qquad a\circ a^-
    = - a + a,
 \end{align*}
for all $a,b,c\in S$. We call $(S,+)$ and $(S,\circ )$ the \emph{additive semigroup} and the \emph{multiplicative semigroup} of $S$, respectively.
\end{defin}
\noindent Any weak brace $S$ is an inverse semi-brace. Indeed, we have
\begin{align*}
    -a + a\circ b
  = -a + a - a + a\circ b
  = a\circ a^- -a+ a\circ b
  = a\circ\left(a^- + b\right),
\end{align*}
for all $a,b\in S$. 
\noindent The trivial and the almost trivial inverse semi-braces in \cref{exs-1} are easy instances of weak braces obtained starting from a Clifford semigroup.
On the other hand, not every inverse semi-brace is a weak brace, see for instance \cite[Example 1]{CaMaSt21}.

\medskip

Clearly, the sets of the idempotents $E(S,+)$ and $E(S,\circ)$ coincide. As a direct consequence, since inverse semigroups with exactly one idempotent are groups (cf. Proposition 4 of Section 1.4 in \cite{La98}), it holds that $(S,+)$ is a group if and only if $(S,\circ)$ is a group.
Furthermore, $(S, +)$ is an idempotent semigroup if and only if $(S, \circ)$ is an idempotent semigroup.
\begin{prop}\label{idempotent}
    Let $S$ be a weak brace. Then, $(S, +)$ is an inverse monoid with unit $0$ if and only if $(S,\circ)$ is an inverse monoid  with unit $0$.
\begin{proof}
Initially, we assume that $(S,+)$ is a monoid with unit $0$. If $a \in S$, we have
\begin{align*}
    a\circ 0-a+a\circ 0=a \circ \left(0+0\right)=a\circ 0
    \end{align*}
and
\begin{align*}
    -a+a\circ 0-a=a\circ\left(a^-+0\right)-a=a\circ a^--a=-a+a-a=-a,
\end{align*}
thus $a \circ 0=-(-a)=a$. Moreover, since $0\in E(S, \circ)$, from what has been just proved, we get
\begin{align*}
    a=a \circ a^-\circ a=a \circ a^-\circ 0\circ a=0 \circ a \circ a^-\circ a=0\circ a,
\end{align*}
because $0$ and $a\circ a^-$ commute. Thus, $(S, \circ)$ is a monoid with unit $0$.\\ 
Now, we suppose that $(S, \circ)$
is a monoid with unit $0$. Since, $0 \in E(S,+)$, we have $-0=0$. If $a \in S$, we obtain
\begin{align*}
    a=0\circ(a-a+a)=0 \circ a -0+ 0\circ\left(-a+a\right)= a+0 -a+a=a+0,
\end{align*}
because $0$ and $-a+a$ commute. On the other hand, we have
\begin{align*}
    a=a-a+a = a-a+0+a
    =0+a-a+a
    =0+a.
\end{align*}
Therefore, $(S,+)$ is a monoid with unit $0$.
\end{proof}
\end{prop}

\medskip

The additive structure of an arbitrary weak brace $S$ is a Clifford semigroup, as we will prove later. Thus, the fact that the trivial and the almost trivial weak braces are defined by means of a Clifford semigroup is essential. 
To show this result on $(S,+)$, we need to give some structural properties, many of which are already known for skew braces, and that we will also use throughout the paper.

\medskip

The next lemma includes a result contained in  \cite[Lemma 1.7]{GuVe17} and establishes a link between the addition and the multiplication of any weak brace.
\begin{lemma}\label{le:prop-meno}
Let $S$ be a weak brace. Then, the following statements hold:
  \begin{enumerate}
      \item $a \circ (-b) = a - a \circ b + a$,
     \item $a\circ b= a+\lambda_a(b)$,
  \end{enumerate}
  for all $a,b \in S$.
  \begin{proof}
 $1.$ \  If $a, b \in S$, at first we check that 
  $a\circ \left(-b\right) = -\left(- a +a\circ b - a\right)$. Indeed,
  \begin{align*}
      &a \circ (-b) +\left(- a +a \circ b-a \right)+ a \circ (-b)\\
      &= a \circ (-b+b)-a+ a \circ (-b)=a\circ(-b+b-b)=a \circ (-b)
  \end{align*}
  and
  \begin{align*}
     &\left(- a +a \circ b-a \right)+a \circ (-b)+\left(- a +a \circ b-a \right)\\
     &=-a + a \circ \left(b-b\right)-a+a\circ b-a
     = -a+a\circ (b-b+b)-a=-a+a\circ b-a,
  \end{align*}
hence we get the claim. \\
$2.$ \ If $a,b\in S$, by $1.$, we have that
\begin{align*}
   a\circ b= a \circ \left(-\left(-b\right)\right)=a-a\circ(-b)+a=a-a+a\circ b-a+a.
   \end{align*}
   Thus, 
   \begin{align*}
    a\circ b=a-a+(a-a+a\circ b-a+a)=a-a+a \circ b= a+\lambda_a(b),
\end{align*}
which is our claim.
  \end{proof}
\end{lemma}

\medskip

The next result concerning the map $\lambda:S\to \End\left(S,+\right)$ is already proved for skew braces, cf. \cite[Proposition 1.9 - Corollary 1.10]{GuVe17}.
\begin{prop}\label{prop:lambda-hom}
  Let $S$ be a weak brace. Then, the map $\lambda_a$ is an endomorphism of $(S,+)$, for every $a \in S$. Moreover, the map $\lambda:S\to \End\left(S,+\right), a\mapsto \lambda_a$ is a homomorphism of the inverse semigroup $\left(S,\circ\right)$ into the endomorphism semigroup of $\left(S,+\right)$.
  \begin{proof}
  Let $a,b,c \in S$. Then, we obtain
  \begin{align*}
   \lambda_a(b+c)=-a + a\circ (b +c)=-a+a\circ b-a+a \circ c=\lambda_a(b)+\lambda_a(c).   
  \end{align*}
  Moreover,  we get
  \begin{align*}
     &\lambda_a(-b)+\lambda_a(b)+\lambda_a(-b)=\lambda_a(-b+b-b)= \lambda_a(-b),\\
     &\lambda_a(b)+\lambda_a(-b)+\lambda_a(b)=\lambda_a(b-b+b)=\lambda_a(b),
  \end{align*}
  thus $-\lambda_a(b)=\lambda_a(-b)$, i.e., $\lambda_a$ is an endomorphism of the inverse semigroup $(S,+)$.
  Finally, by $1.$ in \cref{le:prop-meno}, we have that
  \begin{align*}
      \lambda_{a \circ b}(c)&=-a\circ b+a\circ b\circ c=-a+a\circ(-b)-a+a\circ b\circ c\\
      &=-a+a\circ(-b+b\circ c)=-a+a\circ \lambda_b(c)=\lambda_a\lambda_b(c),
  \end{align*}
  thus we get the claim.
  \end{proof}
  \end{prop}

\medskip

\begin{lemma}\label{prop-gamma-meno}
    Let $S$ be a weak brace. Then, it holds
    $\lambda_a(a^-) = - a$, for every $a \in S$.
\begin{proof}
Let $a\in S$. By $2.$ in \cref{le:prop-meno}, it holds 
$a\circ\left(a^-+\lambda_{a^-}\left(a\right)\right) = a\circ \left(a^-\circ a \right)= a$. Moreover, by \cref{prop:lambda-hom}, we get
   \begin{align*}
       a 
       &= a\circ\left(a^-+\lambda_{a^-}\left(a\right)\right)
       = a\circ a^- - a + a\circ\lambda_{a^-}\left(a\right)
       = a\circ a^-  + \lambda_{a\circ a^-}\left(a\right)\\
       &=a\circ a^- - a\circ a^- + a
       = a\circ a^-+a.
\end{align*}
It follows that
\begin{align*}
    \lambda_a\left(a^-\right)=-a+a\circ a^-=-\left(-a\circ a^-+a\right)=-\left(a\circ a^-+a\right)=-a.
\end{align*}
Therefore, the claim is proved.
\end{proof}
\end{lemma}

As a consequence of \cref{prop-gamma-meno}, if $S$ is a weak brace, we trivially have that $\rho_{a^-}\left(a\right)^-
    = a\circ\left(a^- + a^-\right)
    = \lambda_{a}\left(a^-\right)=-a$, for every $a \in S$.

\bigskip

We highlight that statements in \cref{le:prop-meno}, \cref{prop:lambda-hom}, and \cref{prop-gamma-meno} hold more in general starting from a triple $(S, +, \circ)$ for which $(S,+)$ and $(S, \circ)$ are inverse semigroups and 
$a\circ\left(b + c\right) = a\circ b - a + a\circ c$ holds, for all $a,b,c\in S$.
The additional property $a\circ a^- = - a + a$ allows one to show that the additive semigroup of any weak brace lies in the special class of Clifford semigroups. 
\begin{theor}\label{prop_+_Clifford}
  Let $S$ be a weak brace. Then, the additive structure $\left(S, +\right)$ is a Clifford semigroup.
  \begin{proof}
  In light of \cite[Exercises II.2.14 (i)]{Pe84}, to prove our statement we show that $(S,+)$ is a completely regular semigroup.
  By \cref{prop-gamma-meno} and \cref{le:prop-meno}-$2.$,  if $a \in S$, we obtain
\begin{align*}
    a - a = a + \lambda_a\left(a^-\right) 
    = a\circ a^- = -a + a.
\end{align*}
Therefore, $(S,+)$ is a Clifford semigroup.
  \end{proof}
\end{theor}

\medskip

Below, we give some examples of weak braces. Let us observe that the third one ensures that the multiplicative semigroup of a weak brace is not a Clifford semigroup in general.  
\begin{exs}\label{ex:esempi}\hspace{1mm}
\begin{enumerate}
    \item Let $S,T$ be two Clifford semigroups. Then, $S \times T$ endowed with the following operations
    \begin{align*}
        \left(a,u\right) + \left(b,v\right)
            = \left(ab, vu\right)\qquad
                \left(a,u\right)\circ \left(b,v\right)
            = \left(ab, uv\right),
    \end{align*}
    for all $\left(a,u\right), \left(b,v\right)\in S\times T$, is a weak brace. 
    \smallskip
    \item Let us consider a commutative inverse semigroup $S$ and a group $T$, not necessarily abelian. If $\alpha: S\to\Aut(T), \, a \mapsto \alpha_a$ is a homomorphism of inverse semigroups, then $S\times T$ endowed with the following operations
    \begin{align*}
        \left(a,u\right) + \left(b,v\right)
        = \left(ab, u\alpha_a\left(v\right)\right)\qquad 
        \left(a,u\right)\circ \left(b,v\right)
        = \left(ab, uv\right),
    \end{align*}
    for all $\left(a,u\right), \left(b,v\right)\in S\times T$, is a weak brace.
    \smallskip
    \item Let $S,T$ be two Clifford semigroups. If $\beta: S\to\Aut(T),\, a \mapsto \beta_a$ is a homomorphism of inverse semigroups, then $S\times T$ endowed with the operations defined by
\begin{align*}
    \left(a,u\right) + \left(b,v\right)
    = \left(ab, uv\right)\qquad
    \left(a,u\right)\circ \left(b,v\right)
    = \left(ab, u\beta_a\left(v\right)\right),
    \end{align*}
for all $\left(a,u\right), \left(b,v\right)\in S\times T$, is a weak brace.
\end{enumerate}
\end{exs}

Note that the previous examples include those of skew braces provided by Koch and Truman in \cite[Example 6.1]{KoTr20} and by Guarnieri and Vendramin in \cite[Examples 1.4-1.5]{GuVe17}.\\
Further examples of weak braces will be subsequently presented in the last two sections. Specifically, the third weak brace in \cref{ex:esempi} will be a particular case of a construction we deal with in \cref{sez-4}.

\bigskip

The following propositions are crucial to prove that any weak brace determines a solution. 
\begin{prop}\label{nuove_lambda_rho}
Let $S$ be a weak brace. Then, the following hold:
\begin{enumerate}
    \item $\lambda_a(b)=-a\circ b \circ b^-+a\circ b$,
    \item $\lambda_a\left(b\right) = a\circ b\circ \rho_b\left(a\right)^-$,
    \item $\rho_b\left(a\right)^- = b^-\circ a^- - b^-$,
    \item $\rho_b\left(a\right)
= \lambda_a\left(b\right)^-\circ a \circ b$,
\end{enumerate}
for all $a,b\in S$.
In particular,
\begin{align*}
    \lambda_a(b)=a\circ b \circ b^- \circ \left(a^-+b\right)
    \qquad
    \rho_b(a) = \left(a^-+b\right)^-\circ a^-\circ a\circ b,
\end{align*}
for all $a,b\in S$.
\begin{proof}
$1.$ \ Let $a,b \in S$. Then,
\begin{align*}
    \lambda_a\left(b\right)&=-a+a\circ \left(b-b+b\right)\\
    &=-a+a\circ \left(-b\circ b^-+b\right)\\
    &=-a+a\circ \left(-b\circ b^-\right)-a+a\circ b\\
    &=-a\circ b \circ b^-+a\circ b. &\mbox{by \cref{le:prop-meno}-$1.$}
\end{align*}
$2.$ \ If $a, b\in S$, we have that 
\begin{align*}
\lambda_a\left(b\right)
&=-a\circ b \circ b^-+a\circ b&\mbox{by $1.$}\\
&= -a\circ b+a\circ b-a\circ b \circ b^- + a\circ b&\mbox{by \cref{prop_+_Clifford}}\\
&=a\circ b \circ \left(a\circ b \right)^--a\circ b \circ b^-+a\circ b\\
    &=a\circ b \circ b^-\circ \left(a^-+b\right)\\
    &= a\circ b\circ\rho_b\left(a\right)^-.
\end{align*}
$3.$ \ Let $a,b \in S$. Then, we get
\begin{align*}
    \rho_b\left(a\right)^-
    = b^-\circ \left(a^-+b\right)=b^-\circ a^--b^-+b^-\circ b
    = b^-\circ a^--b^-.
\end{align*}
$4.$ \ If $a,b \in S$, we obtain that
\begin{align*}
    \left(\lambda_a\left(b\right)^-\circ a \circ b\right)^-
    &= \left(a\circ b\right)^-\circ a\circ  \left(a^-+b\right)\\
    &=\left(a\circ b\right)^- -\left(a\circ b\right)^-\circ a+\left(a\circ b\right)^-\circ a\circ b\\
    &=\left(a\circ b\right)^- -b^-\circ a^-\circ a&\mbox{by \cref{prop_+_Clifford}}\\
    &=b^-\circ a^--b^-+b^-\circ \left(-a^-\circ a\right)-b^- &\mbox{by \cref{le:prop-meno}-$1$.}\\
    &=b^-\circ \left(a^-+a^-\circ a\right)-b^-\\
    &=b^-\circ a^--b^-\\
    &=\rho_b\left(a\right)^-, &\mbox{by $3.$}
\end{align*}
which complete our claim.
\end{proof}
\end{prop}
\medskip

As shown in \cite[Proposition 6]{CaCoSt17}, if $S$ is a left cancellative semi-brace, then the map $\rho:S\to \mathcal{T}_S, b\mapsto \rho_b$ is a semigroup anti-homomorphism of the group $\left(S,\circ\right)$ into the monoid $\mathcal{T}_S$. This property also holds in any weak brace.
\begin{prop}\label{prop:rho-anti}
  Let $S$ be a weak brace. Then, the map $\rho:S\to \mathcal{T}_S, b\mapsto \rho_b$  is a semigroup anti-homomorphism of the inverse semigroup $\left(S,\circ\right)$ into the monoid $\mathcal{T}_S$.
  \begin{proof}
 Let $a,b,c \in S$. Then, we get
  \begin{align*}
     \left(\rho_c\rho_b(a)\right)^-
     &=c^- \circ \rho_b(a)^--c^-&\mbox{by \cref{nuove_lambda_rho}-$3.$}\\
     &= c^-\circ\left(b^-\circ a^- - b^-\right) -c^-&\mbox{by \cref{nuove_lambda_rho}-$3.$}\\
      &= c^-\circ b^-\circ a^- - c^- + c^-\circ\left(-b^-\right) - c^-\\
      &= c^-\circ b^-\circ a^- - c^-\circ b^- &\mbox{by \cref{le:prop-meno}-1.}\\
      &= \left(b\circ c\right)^-\circ a^- - \left(b\circ c\right)^-\\
      &= \rho_{b\circ c}\left(a\right)^-&\mbox{by \cref{nuove_lambda_rho}-$3.$}
 \end{align*}
 hence $\rho_{b\circ c} = \rho_c\rho_b$, which is our statement.
 \end{proof}
\end{prop}
\medskip

\begin{rem}
If $S$ is a weak brace, the map $\rho_b$ can be written as 
  \begin{align*}
    \rho_b\left(a\right)
    = \lambda_{\lambda_a\left(b\right)^-}\left(- a\circ b + a + a\circ b\right),
\end{align*}
for all $a,b\in S$, see \cite[Theorem 3.1]{GuVe17}. In fact, we have that
\begin{align*}
    &\lambda_{\lambda_a\left(b\right)^-}\left(- a\circ b + a + a\circ b\right)&\\
    &\qquad= \lambda_{\lambda_a\left(b\right)^-}\left(- \lambda_a\left(b\right) +  a\circ b\right)\\
    &\qquad= -\lambda_{\lambda_a\left(b\right)^-}\left( \lambda_a\left(b\right)\right) + \lambda_{\lambda_a\left(b\right)^-}\left(a\circ b\right)&\mbox{by \cref{prop:lambda-hom}}\\
    &\qquad=\lambda_a\left(b\right)^- + \lambda_{\lambda_a\left(b\right)^-}\left(a\circ b\right) &\mbox{by \cref{prop-gamma-meno}}\\
    &\qquad=\lambda_a\left(b\right)^-\circ a\circ b
    &\mbox{by \cref{le:prop-meno}-$2.$}\\
    &\qquad=\rho_b\left(a\right)&\mbox{by \cref{nuove_lambda_rho}-$4.$}
\end{align*}
which is our claim. 
\end{rem}

\bigskip

\section{The solution associated to a weak brace}
In this section, we show that the map $r$ associated to any weak brace $S$ is a solution. Furthermore, such a map $r$ admits as an inverse the map $r^{op}$, that coincides with the solution associated to the opposite weak brace of $S$, and it holds $rr^{op}=r^{op}r$, namely $r$ is a completely regular element in $\mathcal{T}_{S \times S}$. Finally, we investigate the behavior of the powers of these solutions.

\medskip

Let us begin by introducing the following lemma which allows us to prove the main result of this paper.
\begin{lemma}\label{le:prodlambda-rho}
 Let $S$ be a weak brace. Then, the following identity
\begin{align}\label{eq:lambdarho-circ}
      a \circ b 
      =  \lambda_a(b) \circ \rho_b(a)
  \end{align} 
  holds, for all $a, b \in S$. 
  \begin{proof}
  Let $a,b \in S$. Then, we have
\begin{align*}
    a\circ \left(a^- + b\right)\circ \left(a^- + b\right)^- \circ b&=  a\circ\left(a^- + b - b -a^-\right) \circ b\\
    &=  a\circ\left(a^--a^- + b-b\right)\circ b &\mbox{by \cref{prop_+_Clifford}}\\
    &=  a\circ\left(a^-\circ a + b\circ b^-\right)\circ b\\
    &= \left(a+\lambda_a\left( b \circ b^-\right)\right) \circ b\\
    &=a \circ b \circ b^- \circ b &\mbox{by \cref{le:prop-meno}-$2.$}\\
    &=a \circ b.
\end{align*}
Therefore, this is the desired conclusion.
  \end{proof}
\end{lemma}

\medskip

\begin{theor}\label{teo_soluzione}
Let $S$ be a weak brace. Then, the map associated to $S$ $r:S\times S\to S\times S$ defined by
\begin{align*}
    r\left(a,b\right)
    = \left(a\circ\left(a^- + b\right), \left(a^- + b\right)^-\circ b\right),
\end{align*}
for all $a,b\in S$, is a solution.
\begin{proof}
Given a set $S$, it is a routine computation verifying that a map $r:S\times S\to S\times S$ written as $r\left(a,b\right) = \left(\lambda_a\left(b\right), \, \rho_b\left(a\right)\right)$ is a solution if and only if 
\begin{align}
     &\label{first} \lambda_a\lambda_b(c)=\lambda_{\lambda_a\left(b\right)}\lambda_{\rho_b\left(a\right)}\left(c\right)\\
    &  \label{second}\lambda_{\rho_{\lambda_b\left(c\right)}\left(a\right)}\rho_c\left(b\right)=\rho_{\lambda_{\rho_b\left(a\right)}\left(c\right)}\lambda_a\left(b\right)\\
      &\label{third}\rho_c\rho_b(a)=\rho_{\rho_c\left(b\right)}\rho_{\lambda_b\left(c\right)}\left(a\right),
  \end{align} 
for all $a,b,c \in S$. 
Thus, we prove our statement by showing that these last equalities hold for the map $r$ associated to the weak brace $S$.\\
By \cref{prop:lambda-hom} and by \cref{prop:rho-anti}, the maps $\lambda$ and $\rho$ are a homomorphism and an anti-homomorphism, respectively. Thus, by \eqref{eq:lambdarho-circ}, the relations in \eqref{first} and \eqref{third} follow. \\
Finally, since by \cref{nuove_lambda_rho}, we have
  \begin{align}\label{eq:lr}
      \lambda_a(b)=a\circ b \circ b^- \circ \left(a^-+b\right) \qquad \text{and} \qquad \rho_b\left(a\right)=\lambda_a\left(b\right)^-\circ a \circ b,
  \end{align}
  we get
  \begin{align*}
     &\lambda_{\rho_{\lambda_b\left(c\right)}\left(a\right)}\rho_c\left(b\right)\\
     &\quad=\rho_{\lambda_b\left(c\right)}\left(a\right)\circ \rho_c\left(b\right)\circ \rho_c\left(b\right)^-\circ \left(\rho_{\lambda_b\left(c\right)}\left(a\right)^-+\rho_c\left(b\right)\right)&\mbox{by \eqref{eq:lr}}\\
     &\quad=\left(\lambda_a\lambda_b\left(c\right)\right)^-\circ a \circ \lambda_b\left(c\right)\circ \rho_c\left(b\right)\circ \rho_c\left(b\right)^-\circ \left(\rho_{\lambda_b\left(c\right)}\left(a\right)^-+\rho_c\left(b\right)\right)&\mbox{by \eqref{eq:lr}}\\
     &\quad=\left(\lambda_a\lambda_b\left(c\right)\right)^-\circ a \circ b\circ c\circ \left(\rho_{\rho_c\left(b\right)}\rho_{\lambda_b\left(c\right)}\left(a\right)\right)^- &\mbox{by \eqref{eq:lambdarho-circ}}\\
     &\quad=\left(\lambda_{\lambda_a\left(b\right)}\lambda_{\rho_b\left(a\right)}\left(c\right)\right)^- \circ a \circ b\circ c\circ \left(\rho_{c}\rho_{b}\left(a\right)\right)^-&\mbox{by \eqref{first} and \eqref{third}}\\
     &\quad=\left(\lambda_{\lambda_a\left(b\right)}\lambda_{\rho_b\left(a\right)}\left(c\right)\right)^- \circ \lambda_a(b) \circ \rho_b(a)\circ c\circ c^-\circ  \left(\rho_b(a)^-+c \right)&\mbox{by \eqref{eq:lambdarho-circ}}\\
    &\quad=\left(\lambda_{\lambda_a\left(b\right)}\lambda_{\rho_b\left(a\right)}\left(c\right)\right)^- \circ \lambda_a\left(b\right) \circ \lambda_{\rho_b\left(a\right)}\left(c\right)&\mbox{by \eqref{eq:lr}}\\
    &\quad=\rho_{\lambda_{\rho_b\left(a\right)}\left(c\right)}\lambda_a\left(b\right),&\mbox{by \eqref{eq:lr}}
  \end{align*}
  thus \eqref{second} is satisfied. Therefore, $r$ is a solution.
\end{proof}
\end{theor}

\medskip

Now, to show that the solution $r$ associated to any weak brace $S$ is a completely regular map, we introduce the notion of opposite weak brace of $S$, which is consistent with that given by Koch and Truman \cite[Proposition 3.1]{KoTr20} in the context of skew braces. 

\begin{prop}\label{prop:opposite}
  Let $(S,+,\circ)$ be a weak brace and define $a+^{op}b:= b + a$, for all $a,b \in S$. Then, $S^{op}:=\left(S,+^{op}, \circ\right)$ is a weak brace, which we call the \emph{opposite weak brace of $S$}.
  \begin{proof}
  Clearly, $\left(S, +^{op}\right)$ is an inverse semigroup with $-^{op}a=-a$, for every $a \in S$.
  Thus,
  \begin{align*}
      -^{op}a +^{op} a = a - a 
      = a\circ a^-.
  \end{align*}
  Moreover, if $a,b,c\in S$, we get 
  \begin{align*}
      a \circ \left(b +^{op}c\right)=a \circ \left(c+b\right)=a \circ c-a+a\circ b=a \circ b -^{op}a+^{op}a\circ c,
  \end{align*}
  thus $S^{op}$ is a weak brace.
  \end{proof}
\end{prop}

Simple examples of opposite weak braces are the trivial and the almost trivial ones.

\medskip

\begin{rem}\label{rop}
Given a weak brace $S$, let us observe that the solution $r^{op}$ associated to $S^{op}$ can be written as
    \begin{align*}
      r^{op}\left(a,b\right)
      = \left(a\circ b-a, \ \left(a\circ b-a\right)^-\circ a \circ b\right)
      = \left(\rho_{a^-}\left(b^-\right)^-, \ \lambda_{b^-}\left(a^-\right)^-\right),
  \end{align*}
  for all $a,b \in S$. 
  Indeed, by \cref{nuove_lambda_rho}-$3.$,
   \begin{align*}
      \lambda_a\left(b\right)^{op} =
      a\circ \left(b+a^-\right)= a \circ b-a +a\circ a^-=a \circ b-a = \rho_{a^-}\left(b^-\right)^-
  \end{align*}
  and 
  \begin{align*}
    &\rho_b\left(a\right)^{op}
    = \left(\lambda_a\left(b\right)^{op}\right)^-\circ a\circ b&\mbox{by  \cref{nuove_lambda_rho}-$4.$}\\
    &= \rho_{a^-}\left(b^-\right)\circ a\circ b\\
    &= \lambda_{b^-}\left(a^-\right)^-,&\mbox{by \cref{nuove_lambda_rho}-$2.$}
    \end{align*}
    for all $a,b\in S$.
  \end{rem}

\medskip

The following theorem illustrates a significant property of the solution $r$ associated to an arbitrary weak brace $S$, namely $r$ has a behavior close to bijectivity. Moreover, this result includes that in \cite[Theorem 4.1]{KoTr20}, where it is proved that the inverse of the bijective solution $r$ associated to a skew brace $S$ is the solution $r^{op}$ associated to the opposite skew brace of $S$.
\begin{theor}\label{theor_rropr}
  Let $S$ be a weak brace, $r$ the solution associated to $S$, and $r^{op}$ the solution associated to the opposite weak brace of $S$. Then, the following hold 
  \begin{align*}
      r\, r^{op}\, r = r, \qquad
      r^{op}\, r\, r^{op} = r^{op}, \qquad \text{and}\qquad rr^{op} = r^{op}r,
  \end{align*}
  namely, $r$ is a completely regular element in $\mathcal{T}_{S\times S}$. Moreover, $rr^{op}$ is a solution.
  In particular, if $S$ is a skew brace, then the solution $r$ is bijective with $r^{-1} = r^{op}$.
 \begin{proof}
    Let $a,b \in S$. Initially, we prove that $rr^{op} = r^{op}r$. Note that, by \cref{prop:lambda-hom}, we have
    \begin{align}\label{eq:somma}
        a\circ b - a\circ b + a
        = a\circ b +\lambda_{a}\left(-b\right)
        = a\circ \left(b-b\right)
        = a\circ b\circ b^-.
    \end{align}
    Thus, the first component of $rr^{op}\left(a,b\right)$ is
\begin{align*}
   -\lambda_{a}\left(b\right)^{op} + \lambda_{a}\left(b\right)^{op}\circ \rho_{b}\left(a\right)^{op}
   &=a-a\circ b+a\circ b &\mbox{by \cref{rop} and \eqref{eq:lambdarho-circ}}\\
   &=a\circ b-a\circ b+a &\mbox{by \cref{prop_+_Clifford}}\\
   &=a\circ b\circ b^-.&\mbox{by \eqref{eq:somma}}
\end{align*}
Moreover, by \cref{nuove_lambda_rho}-$4.$ and by \eqref{eq:lambdarho-circ}, the second component of $rr^{op}(a,b)$ is equal to
\begin{align*}
   \left(\lambda_{\lambda^{op}_{a}\left(b\right)}\rho^{op}_{b}\left( a \right)\right)^-\circ\lambda^{op}_{a}\left(b\right)\circ \rho^{op}_{b}\left( a \right)
    = \left(a\circ b\circ b^-\right)^-\circ a \circ b=a^-\circ a\circ b.
\end{align*}
    Besides, we obtain that
    \begin{align*}
      r^{op}r\left(a,b\right) 
        &= \left(\lambda_a\left(b\right)\circ \rho_b\left(a\right) - \lambda_a\left(b\right), \ \left(\lambda_a\left(b\right)\circ \rho_b\left(a\right) - \lambda_a\left(b\right)\right)^-\circ \lambda_a\left(b\right)\circ\rho_b\left(a\right) \right)\\
        &= \left(a\circ b - a\circ b + a, \left(a\circ b - a\circ b + a\right)^-\circ a\circ b\right)&\mbox{by \eqref{eq:lambdarho-circ}}\\
        &= \left(a\circ b\circ b^-, \  a^-\circ a\circ b\right).&\mbox{by \eqref{eq:somma}}
    \end{align*}
    Therefore, the map $r^{op}r = rr^{op}$. Moreover, it is a routine computation to verify that $rr^{op}$  satisfies \eqref{first}, \eqref{second}, and \eqref{third}, hence it is a solution. It follows that, by \cref{nuove_lambda_rho}-$1.$, the first component of $r\, r^{op}\, r(a,b)$ is
\begin{align*}
    \lambda_{a\circ b\circ b^-}\left(a^{-}\circ a\circ b \right)
    =-a\circ b\circ b^-+a\circ b=\lambda_a\left(b\right)
\end{align*}
and,  by \cref{nuove_lambda_rho}-$4.$, the second component is equal to
 \begin{align*}
  \rho_{a^-\circ a\circ b}\left(a\circ b\circ b^{-} \right) =
   \lambda_a\left(b\right)^-\circ a \circ b=\rho_b\left(a\right),
 \end{align*}
and so $r\, r^{op}\, r = r$. 
Furthermore, by \cref{rop}, the first component of $r^{op}\,r\,r^{op}(a,b)$ is
 \begin{align*}
     a\circ b\circ b^-\circ a^-\circ a\circ b &- a\circ b\circ b^-= a\circ b - a\circ b\circ b^-\\
     &= a\circ b - a + a\circ \left(-b\circ b\right)^- - a &\mbox{by \cref{le:prop-meno}-$1.$}\\
     &= a\circ \left(b + b\circ b^-\right) - a \\
     &= a\circ b - a=\lambda^{op}_a\left(b\right)
 \end{align*}
 and the second one is
\begin{align*}
    \left(a\circ b - a\right)^-\circ a\circ b\circ b^-\circ a^-\circ a\circ b
    = \left(a\circ b - a\right)^-\circ a\circ b
    = \rho^{op}_b\left(a\right),
\end{align*}
i.e., $r^{op}\,r\,r^{op} = r^{op}$.
Therefore, $r$ is a completely regular map. Finally, if $S$ is a skew brace, by \cite[Thereom 3.1]{GuVe17}, $r$ is bijective, hence we clearly have that $r^{op} = r^{-1}$.
\end{proof}
\end{theor}

\medskip

\begin{rems}\label{remark_lambda_rho_inverse}\hspace{1mm}
\begin{enumerate}
    \item In general, one can note that every solution may have more than one inverse. An example is given by the solution $r$ on a monoid $(S,\circ)$, with $|S|\geq 1$, defined by $r\left(a,b\right) = \left(a\circ b,\, 1\right)$, for all $a,b\in S$. Indeed, $r$ is an idempotent solution (cf. \cite[Examples 1.1]{CaMaSt20}), thus an inverse is $r$ itself.
    Furthermore, the idempotent solution $s\left(a,b\right) = \left(1,\, a\circ b\right)$ is another inverse of $r$, which does not commute with $r$. 
\vspace{1mm}

\item If $S$ is a weak brace, then the solution $r$ associated to $S$ has a behavior which is close to non-degeneracy. Indeed, as observed in \cref{lambdaalambaa-}, the map $\lambda_a$ admits the map $\lambda_{a^-}$ as an inverse, for every $a\in S$. Similarly, by \cref{prop:rho-anti}, the map $\rho_b$ admits $\rho_{b^-}$ as an inverse, for every $b \in S$. Evidently, if $(S, \circ)$ is a Clifford semigroup, such maps are completely regular. In the next section, we will provide a weak brace in \cref{ex_semidirect} for which $\lambda_{a \circ a^-}\neq\lambda_{a^- \circ a}$.
\end{enumerate}
\end{rems}

\medskip

\begin{rem}
Note that the map $rr^{op}$ in \cref{theor_rropr} can be obtained starting from an arbitrary inverse semigroup. Namely, if $(S, \circ)$  is an inverse semigroup, the map $s=rr^{op}$ given by \begin{align*}
        s\left(a,b\right)
        = \left(a \circ b \circ b^{-}, a^{-} \circ a \circ b\right),
    \end{align*}
    for all $a,b\in S$, is a solution that is idempotent.
\end{rem}

\bigskip

Now, we focus on the behavior of the powers of the solutions $r$. We start by proving the following lemma, which is already known for solutions associated to skew braces, see \cite[Lemma 4.13]{SmVe18}. The properties obtained until now allow us to prove it in a similar way even in the more general context.
\begin{lemma}\label{lemma_potenze}
	Let $S$ be a weak brace and $r$ the solution associated to $S$. Then, the following hold
	\begin{align*}
	&r^{2n}\left(a,b\right)
	=\left(-n\left(a \circ b \right)+a+n\left(a \circ b\right), \left(-n\left(a \circ b \right)+a+n\left(a \circ b\right)\right)^- \circ a \circ b\right)\\
	&r^{2n+1}\left(a,b\right)
	=\left(-n\left(a \circ b \right)-a+(n+1)\left(a \circ b\right), \left(-n\left(a \circ b \right)-a+(n+1)\left(a \circ b\right)\right)^- \circ a \circ b\right),
	\end{align*}
	for every $n\in \mathbb{N}$, $n>0$, and for all $a,b \in S$.
	\begin{proof}
		For the sake of simplicity, 
		set $\lambda_a^{(1)}\left(b\right):= \lambda_a\left(b\right)$ and $\rho_b^{(1)}\left(a\right):= \rho_b\left(a\right)$,
		and for every $n\in \mathbb{N}$, $n >0$, write 
		\begin{align*}
		\lambda_a^{(n)}\left(b\right):= \lambda_{\lambda_a^{(n-1)}\left(b\right)}\rho_b^{(n-1)}\left(a\right) \quad \text{and} \quad \rho_b^{(n)}\left(a\right):= \rho_{\rho_b^{(n-1)}\left(a\right)}\lambda_a^{(n-1)}\left(b\right),
		\end{align*}
		for all $a,b \in S$. It is a routine computation to check that
		\begin{align*}
		r^n\left(a,b\right) = \left(\lambda_a^{(n)}\left(b\right), \rho_b^{(n)}\left(a\right)\right)
		\end{align*}
		and
		\begin{align}\label{gamma_eta_n}
		  \lambda_a^{(n)}\left(b\right)\circ \rho_b^{(n)}\left(a\right) = a \circ b
		\end{align}
		hold, for all $a,b \in S$. Furthermore, by \eqref{gamma_eta_n} and by \cref{nuove_lambda_rho}-$4.$, proceeding by a simple induction we obtain
		\begin{align*}
		    \rho_b^{(n)}\left(a\right) = \left(\lambda_a^{(n)}\left(b\right)\right)^- \circ a \circ b,
		\end{align*}
		for all $a,b \in S$. Hence, we only prove that
		\begin{align}
		    \label{gamma_pari}&\lambda_a^{(2n)}(b) = -n\left(a \circ b \right) + a + n\left(a \circ b\right),\\
		    \label{gamma_dispari}&\lambda_a^{(2n+1)}(b) = -n\left(a \circ b \right) - a + (n + 1)\left(a \circ b\right),
		\end{align}
		for all $a,b \in S$, and we proceed by induction on $n>0$.\\ If $a,b \in S$, for $n = 1$, by  \eqref{eq:lambdarho-circ} we get
\begin{align*}
    \lambda_a^{(2)}\left(b\right)=\lambda_{\lambda_a\left(b\right)}\rho_b\left(a\right)=-\lambda_a\left(b\right)+\lambda_a\left(b\right)\circ \rho_b\left(a\right)=- a \circ b +a +a \circ b 
\end{align*}		
		and, 
		\begin{align*}
    \lambda_a^{(3)}\left(b\right)&=\lambda_{\lambda_a^{(2)}\left(b\right)}\rho_b^{(2)}\left(a\right)=\lambda_a^{(2)}\left(b\right) +\lambda_a^{(2)}\left(b\right)\circ\rho_b^{(2)}\left(a\right)\\
    &=\left(- a \circ b +a +a \circ b \right)+a \circ b &\mbox{by \eqref{gamma_eta_n}}\\
    &=-a \circ b-a+2\left(a \circ b\right).
\end{align*}
Assume the claim follows for some $n >1$. Then, if $n$ is even, by induction hypothesis, we have
\begin{align*}
    \lambda_a^{(2n+1)}(b)&=\lambda_{\lambda_a^{(2n)}\left(b\right)}\rho_b^{(2n)}\left(a\right)=-\lambda_a^{(2n)}\left(b\right)+\lambda_a^{(2n)}\left(b\right) \circ \rho_b^{(2n-1)}\left(a\right)\\
    &=-\lambda_a^{(2n)}\left(b\right)+a \circ b &\mbox{by \eqref{gamma_eta_n}}\\
    &= -\left(-n\left(a \circ b \right)+a+n\left(a \circ b\right)\right)+a \circ b\\
    &=-n\left(a \circ b \right)-a+(n+1)\left(a \circ b\right),\end{align*}
hence \eqref{gamma_dispari} follows. Analogously, if $n$ is odd, one can see that \eqref{gamma_pari} is satisfied.
	\end{proof}
\end{lemma}

\medskip

In the following proposition, we prove that the solutions associated to a weak brace having a commutative semigroup as additive structure are cubic, i.e., $r^3 = r$. 
Clearly, if $S$ is a brace, since $r$ is bijective we obtain $r^2 = \id_{S\times S}$, see \cite{Ru07}. 
\begin{prop}
  Let $S$ be a weak brace and $r$ the solution associated to $S$. Then, if the inverse semigroup $\left(S,+\right)$ is commutative it follows that 
  $r^3 =r$.
  \begin{proof}
  Let $a,b \in S$. Then, by \cref{lemma_potenze}, we have that
  \begin{align*}
      \lambda^{(3)}_a\left(b\right)=-a \circ b-a+2\left(a \circ b\right)=-a+a\circ b-a\circ b+a\circ b=-a+a\circ b=\lambda_a\left(b\right),
  \end{align*}
  hence, by $4.$ in \cref{nuove_lambda_rho}, $\rho^{(3)}_b\left(a\right)=\left(\lambda^{(3)}_a\left(b\right)\right)^-\circ a \circ b=\left(\lambda_a\left(b\right)\right)^-\circ a \circ b=\rho_b(a)$. Therefore, the claim follows.
  \end{proof}
\end{prop}

\medskip

To conclude this section, we give a property of a solution associated to a weak brace $S$ with an idempotent semigroup $(S,+)$.
\begin{prop}
  Let $S$ be a weak brace and $r$ the solution associated to $S$. Then, if the inverse semigroup $\left(S,+\right)$ is idempotent it follows that 
  $r^3 =r^2$.
  \begin{proof}
 If $a,b \in S$, then by \cref{lemma_potenze}, we have that
  \begin{align*}
      \lambda^{(3)}_a\left(b\right)=-a \circ b-a+2\left(a \circ b\right)=-a\circ b-a+a\circ b = \lambda^{(2)}_a\left(b\right),
  \end{align*}
  hence $\rho^{(3)}_b\left(a\right)=\left(\lambda^{(3)}_a\left(b\right)\right)^-\circ a \circ b = \left(\lambda_a\left(b\right)^{(2)}\right)^-\circ a \circ b = \rho^{(2)}_b(a)$. Therefore, the claim follows.
  \end{proof}
\end{prop}

\bigskip

\section{Constructions of weak braces}\label{sez-4}

This section aims to review some of the constructions of inverse semi-braces provided in \cite{CaMaSt21} to obtain new examples belonging to the class of weak braces.

\medskip
To this purpose, it is useful to characterize inverse semi-braces which are also weak braces. 
\begin{prop}\label{prop_equivalenza}
Let $S$ be a non-empty set endowed with two operations $+$ and $\circ$ such that $(S,+)$ and $(S, \circ)$ are semigroups. Then, $S$ is a weak brace if and only if $S$ is an inverse semi-brace such that the following hold:
\begin{enumerate}
    \item $(S,+)$ is inverse,
    \item $a\circ \left(a^- + b\right)= -a + a\circ b$,
\end{enumerate}
for all $a,b\in S$.
\begin{proof}
Initially, note that we have already observed that any weak brace is an inverse semi-brace.
Conversely, assume that $S$ is an inverse semi-brace such that $1.$ and $2.$ are satisfied.  We only have to show that, for every $a\in S$, $a\circ a^- = - a + a$. 
Let us observe that the statements in  \cref{le:prop-meno},  \cref{prop:lambda-hom}, and \cref{prop-gamma-meno} still hold. If $a \in S$, we have that $a \circ a^- \in E(S,+)$. In fact,
\begin{align*}
    a\circ a^-
    = a\circ\left(a^- - a^- + a^-\right)
    = a\circ a^- + a\circ\left(a^- - a^- + a^-\right)
    = a\circ a^- + a\circ a^-.
\end{align*}
Moreover, we obtain
 \begin{align*}
  a=a+a\circ a^-.   
 \end{align*}
Indeed, by \cref{prop-gamma-meno},
$-a=\lambda_a\left(a^-\right)=a\circ \left(a^-+a^-\right)=a\circ a^-+\lambda_a\left(a^-\right)=a\circ a^--a$, hence $a=a-a\circ a^-=a+a\circ a^-$.\\
Now, we show that $a \circ a^-$ is the opposite of $-a+a$. We have that
\begin{align*}
    a\circ a^- - a + a + a\circ a^-= a\circ a^- - a + a= a\circ\left(a^- + a^-\circ a\right)= a\circ a^- 
\end{align*}
and
\begin{align*}
   -a + a + a\circ a^- - a + a=-a+a-a+a=-a+a. 
\end{align*}
Therefore, $S$ is a weak brace.
\end{proof}
\end{prop}

\medskip

Let us begin by examining the construction named \emph{matched product of inverse semi-braces}. We show that any matched product of weak braces gives rise to another one without requiring any additional properties.
\noindent To this purpose, we need the maps $\alpha$ and $\beta$ in \cite[Definition 10]{CaMaSt21}, which allow for obtaining such a new weak brace having multiplicative semigroup isomorphic to a Zappa product of the starting inverse semigroups.
Hereinafter, for the ease of the reader, given two weak braces $S$ and $T$, we use the letters $a, b, c$ for the elements of $S$ and $u, v, w$ for the elements of $T$. 
\begin{defin}\label{def:mps-inv}
    Let $S$ and $T$ be two weak braces, $\alpha: T \to \Aut\left(S\right)$ a homomorphism of inverse semigroups from $\left(T,\circ \right)$ into the automorphism group of $\left(S,+\right)$, and $\beta:S\to \Aut\left(T\right)$ a homomorphism of inverse semigroups from $\left(S,\circ \right)$ into the automorphism group of $\left(T,+\right)$ such that
    	\begin{align}
    		\label{eq:mps1}
    		\alphaa{}{u}{\left(\alphaa{-1}{u}{\left(a\right)}\circ  b\right)} = a\circ  \alphaa{}{\beta^{-1}_{a}{\left(u\right)}}{\left(b\right)}\qquad\quad 
    		\beta_{a}\left(\beta^{-1}_{a}\left(u\right)\circ  v\right) = u\circ  \beta_{\alphaa{-1}{u}{\left(a\right)}}\left(v\right)
    		\end{align}
    		\begin{align}\label{eq:mps-idemp}
    		\alphaa{}{u}{\left(\alphaa{-1}{u}{\left(a\right)}\circ  a\right)} = a, \, \, \
    		\beta_{a}\left(\beta^{-1}_{a}\left(u\right)\circ  u\right) = u\, \Longrightarrow
    		\alpha_{u}\left(a\right) = a, \ \beta_{a}\left(u\right) = u
    	\end{align}
    	hold, for all $a,b \in S$ and $u,v \in T$. Then, $(S,T,\alpha,\beta)$ is called a \emph{matched product system of weak braces}.
    \end{defin}
    
    \medskip
    
    \begin{theor}[cf. \cite{CaMaSt21}, Theorem 12] \label{th:matched-inv-semi}
	Let $\left(S,T,\alpha,\beta\right)$ be a matched product system of weak braces. Then, $S\times T$ with respect to 
		\begin{align*}
		\left(a,u\right)+\left(b,v\right) &:=\left(a+b,u+v\right)\\
		\left(a,u\right)\circ \left(b,v\right) &:= \left(\alphaa{}{u}{\left(\alphaa{-1}{u}{\left(a\right)}\, \circ b\right)},\beta_{a}\left(\beta^{-1}_{a}\left(u\right)\,\circ  v\right)\right),
		\end{align*}
		for all $\left(a,u\right), \left(b,v\right) \in S \times T$, is a weak brace, called the \emph{matched product of $S$ and $T$ (via $\alpha$ and $\beta$)} and denoted by $S\bowtie T$.
		\begin{proof}
		    By \cite[Theorem 12]{CaMaSt21}, we have that $S\bowtie T$ is an inverse semi-brace. Moreover, the additive structure is clearly an inverse semigroup since $\left(S\times T, +\right)$ is the direct product of $\left(S, +\right)$  and $\left(T, +\right)$. 
		    Thus, we only need to show that condition $2.$ in \cref{prop_equivalenza} is satisfied.\\	Let $\left(a,u\right), \left(b,v\right)\in S \times T$. Then, we obtain
		    \begin{align*}
		    -\left(a,u\right)&+\left(a,u\right)\circ \left(b,v\right)\\
		    &=\left(-a,-u\right)+\left(\alphaa{}{u}{\left(\alphaa{-1}{u}{\left(a\right)}\, \circ b\right)},\beta_{a}\left(\beta^{-1}_{a}\left(u\right)\,\circ  v\right)\right)\\
		    &=\left(-a+\alphaa{}{u}{\left(\alphaa{-1}{u}{\left(a\right)}\, \circ b\right)},-u+\beta_{a}\left(\beta^{-1}_{a}\left(u\right)\,\circ  v\right)\right)\\
		    &=\left(-a+ a\circ  \alphaa{}{\beta^{-1}_{a}{\left(u\right)}}{\left(b\right)} , \, -u+u\circ  \beta_{\alphaa{-1}{u}{\left(a\right)}}\left(v\right)\right) &\mbox{by \eqref{eq:mps1}}\\
		  &=\left(\lambda_a\alphaa{}{\beta^{-1}_{a}{\left(u\right)}}{\left(b\right)}, \, \lambda_u\beta_{\alphaa{-1}{u}{\left(a\right)}}\left(v\right)\right)
		\end{align*}
		and 
		\begin{align*}
		    &\left(a,u\right)\circ \left(\left(a,u\right)^- + \left(b,v\right)\right)\\
		    &\quad=\left(a,u\right)\circ
		    \left(\left(\alpha^{-1}_{\beta^{-1}_a\left(u\right)}\left(a^-\right),\, \beta^{-1}_{\alpha^{-1}_u\left(a\right)}\left(u^-\right)\right) + \left(b,v\right)\right)\\
		    &\quad= \left(\alpha_u\left(\alpha^{-1}_u\left(a\right)\circ\left(\alpha^{-1}_{\beta^{-1}_a\left(u\right)}\left(a^-\right) + b \right)\right),\,
		    \beta_a\left(\beta^{-1}_a\left(u\right)\circ\left(\beta^{-1}_{\alpha^{-1}_u\left(a\right)}\left(u^-\right) + v\right)\right)
		    \right)\\
		    &\quad= \left(a\circ\left(a^- + \alpha_{\beta^{-1}_a\left(u\right)}\left(b\right)\right),
		    u\circ\left(u^- + \beta_{\alpha^{-1}_u\left(a\right)}\left(v\right)\right)\right)&\mbox{by \eqref{eq:mps1}}\\
		    &\quad=\left(\lambda_a\alphaa{}{\beta^{-1}_{a}{\left(u\right)}}{\left(b\right)}, \, \lambda_u\beta_{\alphaa{-1}{u}{\left(a\right)}}\left(v\right)\right).
		\end{align*}
		Therefore, by \cref{prop_equivalenza}, $S\times T$ is a weak brace.
		\end{proof}
\end{theor}

\medskip

As shown in \cite[Theorem 15]{CaMaSt21}, the solution $r$ associated to any matched product $S \bowtie T$ of two weak braces $S$ and $T$ is exactly the matched product
of the solutions $r_S$ and $r_T$ associated to $S$ and $T$, respectively, and it is given by
		\begin{align*}
			&r\left(\left(a, u\right), 
			\left(b, v\right)\right) := 
			\left(\left(\alphaa{}{u}{\lambdaa{\bar{a}}{\left(b\right)}},\, \beta_a\lambdaa{\bar{u}}{\left(v\right)}\right),\ \left(\alphaa{-1}{\overline{U}}{\rhoo{\alphaa{}{\bar{u}}{\left(b\right)}}{\left(a\right)}},\,  \beta^{-1}_{\overline{A}}\rhoo{\beta_{\bar{a}}\left(v\right)}{\left(u\right)}\right) \right),
		\end{align*}
		\noindent where we set
		\begin{center}
		   $\bar{a}:=\alphaa{-1}{u}{\left(a\right)}$, \,\,$\bar{u}:= \beta^{-1}_{a}\left(u\right)$,\,\, $A:=\alphaa{}{u}{\lambdaa{\bar{a}}{\left(b\right)}}$,\, $U:=\beta_a\lambdaa{\bar{u}}{\left(v\right)}$,\,\, $\overline{A}:=\alphaa{-1}{U}{\left(A\right)}$,\,\, $\overline{U}:= \beta^{-1}_{A}\left(U\right)$,
		\end{center}
		for all $\left(a,u\right),\left(b,v\right)\in S\times T$ (cf. \cite[Theorem 14]{CaMaSt21}). 

\medskip

The semidirect product of two weak braces, cf. \cite[Corollary 17]{CaMaSt21}, is a particular case of the matched product.
Specifically, if $\left(S, T, \alpha, \beta\right)$ is a  matched product system, we consider $\beta_a = \id_T$, for every $a\in S$. Analogously, one can consider the case $\alpha_u = \id_S$, for every $u\in T$. A simple example of such a case is $3.$ in \cref{ex:esempi}, regarding both the Clifford semigroups $S$ and $T$ as trivial weak braces. 
In this way, the inverse semigroup $\left(S\times T, \circ\right)$ is exactly the semidirect product of the inverse semigroups $\left(S,\circ \right)$ and $\left(T,\circ \right)$ via $\alpha$ (or via $\beta$), in the sense of
\cite{Ni83} and \cite{Pr86}. Namely, 
this is the Zappa product of the two semigroups $(S,\circ)$ and $(T,\circ)$ with $^{u}a = \sigma\left(u\right)\left(a\right) = \alpha_u\left(a\right)$  and $u^a = \beta_a\left(u\right) = u$, for all $a\in S$ and $u\in T$. In particular, the multiplication $\circ$ on $S \times T$ is given by
\begin{align*}
(a,u) \circ (b,v) = \left(a\circ \, ^{u}b, \, u\circ v \right),
\end{align*}
for all $(a,u), (b,v)\in S\times T$. 

\medskip

The following is an instance of weak brace obtained starting from the semidirect product of Clifford semigroups, cf. \cite[Example 9]{CaMaSt21}.
\begin{ex}\label{ex_semidirect}
    Let $X:=\{1,x,y\}$, $S$ the upper semilattice on $X$ with join $1$, and $T$ the commutative inverse monoid on $X$ with identity $1$ in which they hold $x\circ x = y\circ y= x$  and $x\circ y = y$. 
    Consider the trivial weak braces on $S$ and $T$, respectively.
    If $\tau$ is the automorphism of $S$ given by the transposition $\tau := (x\, y)$, then the map $\sigma:T\to \Aut(S)$ given by $\sigma(1) = \sigma(x) = \id_S$ and $\sigma(y) = \tau$, 
is a homomorphism from $\left(T,\circ\right)$ into the
automorphism group of the weak brace $S$.
 Therefore, by \cref{th:matched-inv-semi}, $S \times T$ is the semidirect product of $S$ and $T$.\\
 In addition, as anticipated in \cref{remark_lambda_rho_inverse}, there exists an element $(a,u) \in S \times T$ such that $\lambda_{(a,u)\circ (a,u)^-}
 \neq \lambda_{(a,u)^-\circ (a,u)}$. Indeed,
 $\lambda_{(y,y)\circ (y,y)^-}\left(y,1\right) = \lambda_{(y,x)}\left(y,1\right) = \left(y,x\right)$ and
 $\lambda_{(y,y)^-\circ (y,y)}\left(y,1\right) = \lambda_{(x,x)}\left(y,1\right)
 = \left(1, x\right)$.
\end{ex}

\bigskip

Now, we show how to obtain a new weak brace involving the construction of the \emph{double semidirect product of inverse semi-braces}, cf. \cite[Theorem 19]{CaMaSt21}. In this case, to obtain an inverse semigroup $(S\times T,+)$, we need that the codomain of the map $\delta$ is $\Aut(S)$, cf. \cite[Theorem 4]{Wa15}. Moreover, the additional condition \eqref{eq:semibrace-Clifford} ensures that it also is a Clifford semigroup, a necessary condition by \cref{prop_+_Clifford}.
\begin{theor}\label{th_double}
  	Let $S$ and $T$ be two weak braces,  $\sigma: T\to \Aut\left(S\right)$ a homomorphism from $\left(T,\circ\right)$ into the automorphism group of the weak brace $S$, with ${}^u a:= \sigma(u)(a)$, for all $a \in S$ and $u \in T$, and $\delta:S\to \Aut\left(T\right)$ an anti-homomorphism from $\left(S, +\right)$ into the automorphism group of $\left(T, +\right)$, with $u^a:=\delta(a)(u)$, for all $a \in S$ and $u \in T$. If the following conditions
	\begin{align}\label{eq:semibrace-Clifford}
	&\left(u-u\right)^a=u-u\\
	\label{eq:semibrace-sigma-delta}
        &\left(u\circ v\right)^{\lambda_{a}\left({}^ub\right)}
		+ u\circ \left(\left(u^-\right)^b
		+ w\right)
		= u \circ \left(v^b +w\right)
	\end{align}
		hold, for all $a,b \in S$ and $u,v, w\in T$, then $S \times T$ with respect to
		\begin{align*}
		 \left(a,u\right)+\left(b,v\right)&:=\left(a+b, \, u^b+v\right)\\
		 \left(a,u\right)\circ \left(b,v\right)&:=\left(a \circ \ ^{u}{b}, \, u\circ v\right),
		\end{align*}
		for all $\left(a,u\right), \left(b,v\right) \in S \times T$,
		is a weak brace. We call such a weak brace the \emph{double semidirect product of $S$ and $T$ (via $\sigma$ and $\delta$)}.
		\begin{proof}
		By \cite[Theorem 19]{CaMaSt21}, we have that $S\bowtie T$ is an inverse semi-brace. Moreover, the additive structure is an inverse semigroup since it is the semidirect product of the semigroups $(S,+)$ and $(T,+)$. Thus, we only show that the condition $2.$ in \cref{prop_equivalenza} is satisfied. If $\left(a,u\right), \left(b,v\right)\in S \times T$, by setting $v=u^-\circ u$ in \eqref{eq:semibrace-sigma-delta}, we obtain
		\begin{align*}
		    u^{\lambda_{a}\left({}^ub\right)}
		+ u\circ \left(\left(u^-\right)^b
		+ w\right)
		&= u \circ \left(\left(u^-\circ u\right)^b +w\right)\\
		&=u \circ \left(\left(u^- - u^-\right)^b +w\right)\\
		&=u \circ \left(u^- - u^-+w\right) &\mbox{by \eqref{eq:semibrace-Clifford}}\\
		&=u \circ \left(u^-\circ u +w\right)\\
		&=u + \lambda_u \left( v\right) &\mbox{by \cref{prop_semi_inversi}-$4.$}\\
		&=u \circ v &\mbox{by \cref{le:prop-meno}-$2.$}
		\end{align*}
		Hence, we get
		\begin{align*}
		    \left(-u\right)^{\lambda_a \left( ^{u}{b}\right)} +u\circ v
		    &= \left(-u\right)^{\lambda_a \left( ^{u}{b}\right)} + u^{\lambda_a \left( ^{u}{b}\right)} + u\circ\left(\left(u^-\right)^b + v\right)\\
		    &=\left(-u + u\right)^{\lambda_a \left( ^{u}{b}\right)} + u\circ\left(\left(u^-\right)^b + v\right)\\
		    &= - u + u + u\circ\left(\left(u^-\right)^b + v\right)&\mbox{by \eqref{eq:semibrace-Clifford}}\\
		    &= - u + u + u\circ\left(u^-\right)^b - u + u\circ v\\
		    &= u\circ\left(u^-\right)^b - u + u\circ v\\
		    &= u\circ\left(\left(u^-\right)^b + v\right).
		\end{align*}
		Thus, it follows that
		    \begin{align*}
		     -\left(a,u\right)+\left(a,u\right)\circ \left(b,v\right)&=\left(-a,-u^{-a}\right)+\left(a \circ \ ^{u}{b}, \, u\circ v\right)\\  
		     &=\left(-a+a \circ \ ^{u}{b},\, \left(-u^{-a}\right)^{a \circ \ ^{u}{b}}+u \circ v\right)\\
		     &=\left(\lambda_a\left( \ ^{u}{b}\right),\, \left(-u\right)^{\lambda_a \left( ^{u}{b}\right)}+u \circ v\right)\\
		     &= \left(\lambda_a\left({}^{u}{b}\right),\, u\circ\left(\left(u^-\right)^b + v\right)\right)
		    \end{align*}
		    and
		    \begin{align*}
		        &\left(a,u\right)\circ\left(\left(a,u\right)^- + \left(b,v\right)\right)
		        = \left(a,u\right)\circ\left(\left({}^{u^-}{a^-},u^-\right)+\left(b,v\right)\right)\\
		        &= \left(a,u\right)\circ\left({}^{u^-}{a^-}+b, \left(u^-\right)^{b} + v\right)\\
		        &= \left(a\circ {}^{u}\left({}^{u^-}{a^-} + b \right), u\circ\left(\left(u^- \right)^b + v\right)\right)\\
		        &= \left(a\circ\left(a^- + {}^{u}b \right), u\circ\left(\left(u^- \right)^b + v\right)\right)\\
		        &= \left(\lambda_a\left({}^{u}b\right), u\circ\left(\left(u^- \right)^b + v\right)\right).
		    \end{align*}
		     Therefore, by \cref{prop_equivalenza}, $S\times T$ is a weak brace.
		\end{proof}
\end{theor}

\medskip

By \cite[Theorem 23]{CaMaSt21}, we have that the solution $r$ associated to a double semidirect product of $S$ and $T$ via $\sigma$ and $\delta$ is given by
\begin{align*}
        r\left(\left(a,u\right), \left(b,v\right)\right)
        =
        \left(\left(\lambda_a\left(^{u}{b}\right), u\circ\Omega_{u,v}^{b}\right),
        \left(^{\left(\Omega_{u,v}^{b}\right)^{-}u^{-}}{\rho_{^u b}\left(a\right)}, \, \left(\Omega_{u,v}^{b}\right)^{-}\circ v\right)\right),
    \end{align*}
    for all $\left(a,u\right), \left(b,v\right)\in S\times T$, with \, $\Omega_{u,v}^{a}
    := \left(u^{-}\right)^a+v$, for all $a\in S$, $u,v\in T$. 

\bigskip

\section{Weak braces obtained from factorizable inverse monoids}

In this section, we give a class of examples of weak braces. Specifically, we extend the construction provided in \cite[Example 1.6]{GuVe17} and \cite[Theorem 3.3]{SmVe18}, involving exactly factorizable inverse monoids. 
Finally, we show that any exactly factorizable group determines two skew braces which, in general, are not isomorphic. 
\medskip

\noindent A semigroup $(S, +)$ is said to be \emph{factorizable} if $S=U+V$, where $U$ and $V$ are subsemigroups of $S$ \cite{To69}. The pair $(U,V)$ is called a \emph{factorization} of $S$, with \emph{factors} $U$ and $V$. Moreover, a factorization $(U,V)$ of $S$ is called \emph{exact} (or \emph{univocal}) if any element $a\in S$ can be written in a unique way as $a = u_a + v_a$, with $u_a\in U$ and $v_a\in V$. In such a case, it holds that $U\cap V= \{0\}$, with $0$ both right identity of $U$ and left identity of $V$, cf. \cite[Theorem 2]{Ca87}. Clearly, if $U$ and $V$ are Clifford semigroups, then $0$ is the identity of $U$ and $V$ and, consequently, $(S, +)$ is a monoid.

\medskip

Let us observe that, in general, if $S$ is a monoid admitting an exact factorization into two Clifford submonoids, then $S$ is not necessarily a Clifford monoid.
However, below we illustrate how to obtain easy examples of Clifford semigroups admitting an exact factorization, involving the semidirect product of a group and a Clifford monoid. 
\begin{ex}
Let $S$ be a group, $T$ a Clifford monoid, and
$\sigma:T\to\Aut\left(S\right)$ a homomorphism. Then, $S\rtimes_{\sigma} T$ is a Clifford semigroup. Indeed, by \cite[Theorem 6]{Pr86}, it is an inverse semigroup. Moreover, it is easy to check that
$E\left(S\rtimes_{\sigma} T\right) = 
\{\left(0,e\right) \, | \, e\in E\left(T\right)\}$.
Hence, if $(a, u) \in S \times T$ and $e\in E\left(T\right)$,
\begin{align*}
    \left(a,u\right) + \left(0,e\right) 
    = \left(a,ue\right)
    = \left(a,eu\right)
    = \left(0,e\right) +\left(a,u\right), 
\end{align*}
i.e., $S\rtimes_{\sigma} T$ is a Clifford semigroup. Furthermore, the factorization of $S\rtimes_{\sigma} T$ is exact.       
Note that, if $T$ is an idempotent semigroup, then the semidirect product is direct.
\end{ex}

We refer to survey by FitzGerald \cite{Fi10} for methods that allow to construct factorizable inverse monoids.

\bigskip

The following theorem is the main result of this section which shows how to obtain a weak brace starting from an exact factorization of a Clifford semigroup.
\begin{theor}\label{prop_Cliff_semibrace}
  Let $(S,+)$ be a Clifford semigroup and $(U,V)$ an exact factorization of $S$ into two Clifford subsemigroups. Define on $S$ the operation given by
  \begin{align*}
      a\circ b:= u_a + b + v_a,
  \end{align*}
  for all $a,b\in S$. Then, $S_{\circ}:=\left(S,+,\circ\right)$ is a weak brace and its multiplicative semigroup is isomorphic to the additive semigroup $U\times V$.
  Moreover, $U$ and $V$ are the trivial and the almost trivial weak braces, respectively.
  \end{theor}
\begin{proof}
  Initially, let us observe that $\left(S, \circ\right)$ trivially is a semigroup. 
  Moreover, $(S,\circ)$ is a completely regular semigroup. Indeed, if $a\in S$ and  $x:= -u_a - v_a$, we get
  \begin{align*}
      a\circ x\circ a= a\circ \left(-u_a + a -v_a\right)
      = u_a - u_a + a -v_a + v_a
      = u_a + v_a= a,
  \end{align*}
  and $a\circ x = u_a-u_a-v_a+v_a=-u_a+u_a+v_a-v_a=x\circ a$.
  Besides, $(S,\circ)$ is a Clifford semigroup. In fact, if $e\in E(S,\circ)$ and $a\in S$, then we have $u_e,v_e\in E(S,+)$ and $a\circ e= u_a+u_e+v_e+v_a=u_e+u_a+v_a+u_e=e\circ a$. 
  In addition, it holds $x=a^-$, since 
  \begin{align*}
      x\circ a\circ x
      = x\circ \left(u_a + x +v_a\right)
      = -u_a + u_a + x +v_a - v_a
      = -u_a - v_a = x.
  \end{align*}
  Now, if $a,b,c \in S$, then we obtain
  \begin{align*}
      a\circ b - a +a\circ c 
      &=u_a+b+v_a -v_a-u_a+u_a+c+v_a\\
      &= u_a-u_a+u_a+b+c+v_a-v_a+v_a\\
      &= a\circ \left(b+c\right)
    \end{align*}
  and  $
      a\circ a^- = u_a-u_a - v_a +v_a = -u_a + u_a -v_a + v_a= -a+a$. Therefore, $S_{\circ}$ is a weak brace.\\ 
    Moreover, it is easy to show that the map $\eta:U\times V\rightarrow S$ given by $\eta(u,v)=u-v$, for all $(u, v) \in U \times V$, is an isomorphism from the additive semigroup $U \times V$ into the multiplicative semigroup $(S, \circ)$.
    Finally, the last part of the statement is trivial.
  \end{proof}
  
\medskip

Let us observe that weak braces $S_{\circ}$ obtained as in the previous theorem are also generalized semi-braces.  
\medskip

In the particular case of factorizable groups, we recover the results given in \cite[Example 1.6]{GuVe17} and \cite[Theorem 3.3]{SmVe18}.
\begin{cor}\label{prop_SmVe18}
Let $(G,+)$ be a group and $(U,V)$ an exact factorization of $G$ into two subgroups. Then, $G_{\circ}$ is a skew brace and its multiplicative group is isomorphic to the additive group $U\times V$. Moreover, $U$ and $V$ are the trivial and the almost trivial skew braces, respectively.
\end{cor}

\smallskip

\begin{rem}
Regarding any element $a=u_a+v_a$ of $G$ as the pair $\left(u_a, v_a\right)$ in $U\times V$, the construction in \cref{prop_SmVe18} is an instance of the double semidirect product of the two skew braces $U$ and $V$ presented in \cref{th_double}, where $\sigma_v = \id_U$, for every $v\in V$, or, equivalently, $(U,+)\unlhd (G,+)$.
Indeed, in such a case,  if $a,b\in U$ and $t, v, w \in V,$ clearly
\begin{align*}
    \left(a,t\right)+\left(b,v\right)
    =\left(a+b, \, t^b + v\right)\qquad
	\left(a,t\right)\circ \left(b,v\right)
	= \left(a + b,v + t\right)
	=\left(a \circ b, \, t\circ v\right),
\end{align*}
and \eqref{eq:semibrace-Clifford} is trivial. Moreover, $\lambda_{a}\left(b\right) = b$ and
\begin{align*}
    \left(t\circ v\right)^{\lambda_{a}\left(b\right) } + t\circ\left(\left(t^-\right)^{b} + w\right)
    = \left(v + t\right)^{b} 
     -t^{b} + w + t 
     = v^b + w + 
    t\circ\left(v^b + w\right), 
\end{align*}
i.e., \eqref{eq:semibrace-sigma-delta} holds.
\end{rem}

\medskip

\begin{rem}
We recall that a skew (left) brace $(S, +, \circ)$ is a \emph{bi-skew (left) brace} if in addition the roles of the sum and the multiplication can be reversed, i.e.,
\begin{align*}
    	a+ (b\circ c)= (a+ b) \circ a^{-} \circ (a+c)
\end{align*}
holds, for all $a,b,c \in S$, cf. \cite{Ch19}. As a direct consequence of \cite[Proposition 7.1]{Ch19}, if $(G,+)$ is a group and $(U,V)$ is an exact factorization of $G$ into two subgroups, then $G_{\circ}$ is a bi-skew brace.
\end{rem}

\medskip

Now, given a group $\left(G, +\right)$ with an exact factorization $(U,V)$, we wonder what happens if we consider the other factorization $(V,U)$ of $G$. Thus, any element $a \in G$ can be also written as $a=v_a'+u_a'$. In a similar way, by exchanging the roles of $U$ and $V$ in \cref{prop_SmVe18}, one can prove the following result.

\begin{prop}\label{prop:bullet}
Let $(G,+)$ be a group and $(V,U)$ an exact factorization of $G$ into two subgroups. Define on $G$ the following operation
\begin{equation*}
	a\bullet b = v_a'+b+u_a',
\end{equation*}
for all $a,b \in G$. Then, $G_{\bullet}:=(G,+,\bullet)$ is a skew brace and its multiplicative group is isomorphic to $V\times U$. Moreover, $V$ and $U$ are the trivial and the almost trivial skew braces, respectively.
\end{prop}

\smallskip

As in the case of $G_{\circ}$, one can check the following result.
\begin{prop}
 Let $(G,+)$ be a group and $(V,U)$ an exact factorization of $G$ into two subgroups, with $V$ a normal subgroup of $G$. Then, $G_{\bullet}$ is a bi-skew brace. 
\end{prop}

\medskip

\noindent In general, we point out that the skew braces $G_{\circ}$ and $G_{\bullet}$ are not isomorphic. 
 Note that, if two skew braces are not isomorphic, the associated solutions can be equivalent. For instance, it is enough to consider the trivial left braces on two abelian not-isomorphic groups having the same order, since the solutions associated coincide with the twist map. Thus, we wonder under which assumptions the skew braces $G_{\circ}$ and $G_{\bullet}$ give rise to non-equivalent solutions.
 
\medskip

The next is an easy example in which the solutions $r_{\circ}$ and $r_{\bullet}$ associated to $G_{\circ}$ and $G_{\bullet}$, respectively, are not equivalent since $G_{\circ}$ is a trivial skew brace and $G_{\bullet}$ is an almost trivial skew brace.
\begin{ex}
Let us consider a group $G$ which is the direct product $G = U+V$ with a non-abelian group $U$ and an abelian group $V$.
Clearly, we have that $u + v = v + u$, for all $u\in U$ and $v\in V$. Then, $G_{\circ}$ is the trivial skew brace, since
\begin{align*}
    a\circ b 
    = u_a + u_b + v_b + v_a
    = u_a + u_b + v_a + v_b
    = u_a + v_a + u_b + v_b
    = a + b,
\end{align*}
for all $a,b\in G$. On the other hand, $G_{\bullet}$ is an almost trivial skew brace. Indeed,
\begin{align*}
    a\bullet b 
    = v_a + v_b + u_b + u_a
    = u_b + u_a + v_a + v_b
    = u_b + a + v_b = b\circ a = b + a,
\end{align*}
for all $a,b\in G$.
Obviously, these skew braces are not isomorphic and the solutions associated to $G_{\circ}$ and $G_{\bullet}$ are not equivalent. In this case, such solutions are 
\begin{align*}
    r_{\circ}\left(a,b\right)
    = \left(b, \ -u_b + a + u_b\right)
    \qquad r_{\bullet}\left(a,b\right)
    =\left(-u_a + b + u_a, \ a\right),
\end{align*}
for all $a,b\in G$, respectively. Moreover, it holds $r_{\circ} = \tau \, r_{\bullet} \,\tau$, where $\tau$ is the twist map.
\end{ex}

\medskip

Now, we focus on exact factorizations having abelian factors. Initially, we observe that if $G=U+V$ is an exact factorization with $U$ and $V$ abelian subgroups of $(G,+)$, then $\left(G,\circ\right)$ is an abelian group, indeed
\begin{align*}
    a\circ b &= u_a + u_b + v_b + v_a
    = u_b + u_a + v_a + v_b\
    = b\circ a,
\end{align*}
for all $a,b \in G$. Similarly, $\left(G,\bullet\right)$ is an abelian group.
\medskip

Recently, a complete classification of skew braces of order $pq$ with $p$ and $q$ primes, up to isomorphism, has been provided in \cite[Theorem p.2]{AcBo20}.
In particular, if $p > q$ and $q\mid p-1$ there are $2q+2$ skew braces, otherwise there is only the trivial one. As far as we have observed, starting from an exact factorization $G=U+V$ different from a direct sum, with $|U|=p$ and $|V|=q$, we obtain that $(G,\circ)$ and $(G,\bullet)$ are abelian groups. Consequently, the skew braces $G_{\circ}$ and $G_{\bullet}$ are among those described in the cases $(ii)$ or $(iii)$ in \cite[Theorem p.2]{AcBo20} with the additive group isomorphic to $\mathbb{Z}_p\rtimes\mathbb{Z}_q$. 
Indeed, by elementary computations, one can check that $(ii)$ and $(iii)$ are the unique skew braces in the classification having abelian multiplicative groups.

\smallskip
Below, we highlight that even in the smallest case, $G_{\circ}$ and  $G_{\bullet}$ are not isomorphic and the solutions  $r_{\circ}$ and  $r_{\bullet}$ are not equivalent.

\begin{ex}\label{ex:S3}
Let us consider the symmetric group $G=\Sym_3$ exactly factorized as $U+V$, where $U= \seq{\left(1 2\right)}$ and $V= \seq{\left(1 2 3\right)}$. Then, $G_{\circ}$ is a skew brace (case $(ii)$), while $G_{\bullet}$ is a bi-skew brace (case $(iii)$) having both the multiplicative groups isomorphic to the cyclic group of order $6$. Specifically, $(G, \circ ) = \seq{(13)} = \seq{(23)}$ and $(G, \bullet) = \seq{(13)} = \seq{(23)}$.
In addition, the solutions $r_{\circ}$ and $r_{\bullet}$ are described by 
\begin{align*}
&\lambda^{\circ}_{\id_{3}} =
\lambda^{\circ}_{\left(12\right)}
= \id_{\Sym_3}\\
&\lambda^{\circ}_{\left(23\right)}
= \left(\left(12\right)\left(13\right)\left(23\right)\right)\\
&\lambda^{\circ}_{\left(123\right)}
= \lambda^{\circ}_{\left(132\right)}
= \lambda^{\circ}_{\left(13\right)}
= \left(\left(12\right)\left(23\right)\left(13\right)\right)\\
&\rho^{\circ}_{\id_3}=\rho^{\circ}_{\left(123\right)}
= \rho^{\circ}_{\left(132\right)}
= \id_{\Sym_3}\\
&\rho^{\circ}_{\left(12\right)}
= \rho^{\circ}_{\left(13\right)}
= \rho^{\circ}_{\left(23\right)}
= \left(\left(13\right)\left(23\right)\right)
\left(\left(123\right)\left(132\right)\right)
\end{align*}
and
\begin{align*}
&\lambda^{\bullet}_{\id_{3}}
= \lambda^{\bullet}_{\left(123\right)}
= \lambda^{\bullet}_{\left(132\right)}
= \id_{\Sym_3}\\
&\lambda^{\bullet}_{\left(12\right)}
= \lambda^{\bullet}_{\left(13\right)}
= \lambda^{\bullet}_{\left(23\right)}
= \left(\left(13\right)\left(23\right)\right)
\left(\left(123\right)\left(132\right)\right)\\
& \rho^{\bullet}_{\id_{3}}
= \rho^{\bullet}_{\left(12\right)} = \id_{\Sym_3}\\
&\rho^{\bullet}_{\left(23\right)}
=\rho^{\bullet}_{\left(123\right)}
= \left(\left(12\right)\left(13\right)\left(23\right)\right)\\
&\rho^{\bullet}_{\left(13\right)}
=\rho^{\bullet}_{\left(132\right)}
= \left(\left(12\right)\left(23\right)\left(13\right)\right).
\end{align*}
These solutions are not equivalent. Indeed, by a routine computation one can check that the unique bijection $f$ from $G$ into itself satisfying $(f \times f)r_{\circ} = r_{\bullet}(f \times f)$ would be the identity on $G$, an absurd.
\end{ex}
    
\bigskip

\bibliography{bibliography}

\end{document}